\documentclass{article} \usepackage{amsmath}
\title{
 On the  lower central factors of groups} \author{ Ted Hurley \thanks{National
    University of Ireland Galway, Ireland. Ted.Hurley@nuigalway.ie}} \date{}
\newcommand{\ga}{\gamma}

\makeatletter
\@ifundefined{theorem}{
\newtheorem{theorem}{Theorem}[section]}{}

\newcommand{\qed}{\hfill{$\text{QED}$}}

\newcommand{\propref}[1]{Proposition~\ref{prop:~1}}

\@ifundefined{proof}{%
\newenvironment{proof}{\noindent\textbf{Proof:}}{\qed}}{} 

\@ifundefined{theorem}{
\newtheorem{theorem}{Theorem}[section]}{}

\@ifundefined{lemma}{
}{}

\@ifundefined{proposition}{
\newtheorem{proposition}{Proposition}}{}

\@ifundefined{corollary}{
\newtheorem{corollary}{Corollary}[section]}{}

\@ifundefined{hypothesis}{
}{}

\@ifundefined{definition}{
}

\@ifundefined{property}{%
}{}
\makeatother
\begin{document}
\maketitle

\begin{abstract} A general method for calculating or 
constructing lower central
  factors of groups is presented. {\it Relative basic
    commutators} are defined.
\end{abstract}

\leftmargin=1pt 
\section{Introduction}
For basic definitions  and results the reader may consult any
standard book   on group theory 
as for example \cite{lynshupp}, \cite{mks} or \cite{hneumann}. 
For a group $G$, the lower central series of $G$ is
defined by:
$ \gamma_1(G) = G$ and $\gamma_{n+1} (G) = [\gamma_n(G),G]$ where $[A,B]$
denotes the subgroup generated by all commutators $[a,b] =
a^{-1}b^{-1}ab$ with $a \in A, b\in B$. The factor group of $G$ by the
normal subgroup $H$ will be denoted by
$\frac{G}{H}$ or $G/H$ as is convenient.

The $n^{th}$ {\em lower central factor of $G$} is the factor group 
$\frac{\gamma_n(G)}{\gamma_{n+1}(G)}$.

Every group is the factor group of a free group. Suppose then 
$G$ is presented as $ G \cong F/R$ where $F$ is free and $R$
is a normal subgroup of $F$. The following result follows from
standard isomorphism theorems. For completeness we include a proof.

\begin{proposition}\label{prop1} Suppose $G \cong F/R$. Then 
$\frac{\gamma_n(G)}{\gamma_{n+1}(G)} \cong
\frac{\frac{\gamma_n(F)}{\gamma_{n+1}(F)}}{\frac{R\cap \gamma_n(F)}{R \cap
    \gamma_{n+1}(F)}}$.
\end{proposition}
\begin{proof} 
Write $\gamma_i$ for $\gamma_i(F)$. From $G \cong F/R$, it follows 
that $\ga_i(G) \cong \frac{\ga_i.R}{R}$. Hence by 
standard isomorphism theorems:

\begin{align*} \frac{\gamma_n(G)}{\gamma_{n+1}(G)} & 
\cong \frac{(\ga_n.R)/R}{(\ga_{n+1}.R)/R} \\ &
\cong \frac{\ga_n.R}{\ga_{n+1}.R} \\ &= \frac{\ga_n.(R.\ga_{n+1})}{
  (R.\ga_{n+1})}  \\ & \cong 
\frac{\ga_n}{(\ga_n\cap (R.\ga_{n+1}))} \\ &= \frac{\ga_n }{(R\cap \ga_n)
.\ga_{n+1}} 
\end{align*} 

Thus \begin{align*} 
\frac{\gamma_n(G)}{\gamma_{n+1}(G)} & \cong \frac{\ga_n/\ga_{n+1}}{(R\cap
 \ga_n).\ga_{n+1}/\ga_{n+1}}  \\ & \cong 
\frac{\ga_n/\ga_{n+1}}{(R\cap \ga_n)/((R\cap \gamma_n) \cap
 \ga_{n+1}))} \\ & 
 \cong \frac{\ga_n/\ga_{n+1}}{(R\cap \ga_n)/(R\cap \ga_{n+1})}\end{align*}
\end{proof}

\medskip

Note also that $ \frac{R \cap \gamma_n}{R\cap \gamma_{n+1} } \cong
\frac{(R\cap\gamma_n)\gamma_{n+1}}{\gamma_{n+1}}$ and this is
a free abelian group as indeed it is a subgroup of the free abelian
group $\ga_n(F)/\ga_{n+1}(F)$.  Thus
if $R\cap \ga_n \mod \ga_{n+1}$ is known, the $n^{th}$ lower
central factor of $G$ is a factor group of the known (see below) free abelian
group $\ga_n/\ga_{n+}$ by the (now known) free abelian group $R\cap \ga_n
\mod \ga_{n+1}$. 

\medskip
\section{Basic commutators}
The structure of $\frac{\gamma_n(F)}{\gamma_{n+1}(F)}$, the $n^{th}$
lower central factor of the free group, is well-known as the free
abelian group on the {\em basic commutators} of weight $n$.

These basic commutators can be defined as follows:

Let $F$ be free on a set $ X$. The basic commutators of weight 1 are
the elements of $X$. Suppose then the basic commutators of weight $ <
n$, with $n\geq 2$,  have been defined and ordered. 
Then a  basic commutator of weight $n$ is a
commutator of the form $[b,c]$ where $b,c$ are basic commutators of
weight $< n$ such that weight $b$ + weight $c = n $ {\it and} if $b = [d,e]$
where $d,e,$ are basic commutators then $ e \leq c$.

This condition that the second component in a basic commutator must be
less than or equal to the next component is a `consequence' of a   
Jacobi-type identity:

$[x,y,z][y,z,x][z,x,y] \in \gamma_{n+1}(G)$

where $[x,y,z] \in \gamma_n(G)$. The Jacobi Identity in groups is a
``$\mod \gamma_{n+1}(G)$'' identity which is  not a direct group identity; 
this indeed makes
group identities and their consequences 
more complicated in general 
-- unless of course one is only interested in working
modulo some term of the lower central series which is often sufficient.  

The most general `free version' of a Jacobi/Witt-Hall  type identity is

$[c,b,a]=[c,b][c,a]^{-1}[b,a,c]^{-1}[b,a]^{-1}[c,b][c,a][c,a,b][b,a]$

and this enables $a$ to be switched to its `correct position' when
it is less than $b$ and $c$ in some ordering. 
\medskip 

Basic commutators were introduced by Philip Hall and
Marshall Hall, see \cite{hall}, \cite{coll}, \cite{mhall}, \cite{hneumann} and
\cite{mks}. The notes  in \cite{hall}  originated in a series of
lectures given by P Hall in 1957 and were available in manuscript 
 form and then 
reproduced with minor changes and additions 
by Queen Mary College College Lecture Notes series in 1969. The notes are
also available within the collected works of  Philip Hall \cite{coll}. 
 
Basic commutators  proved and are still proving very useful in many
areas.  Computations which involve a `collecting
process' of some form follow on from  the collecting process
introduced in 
\cite{hall} and used in the theory of basic commutators and basic
products.  
The connections between free groups, free Lie Algebras, 
free associative rings,  basic commutators in
groups and Lie Algebras, the ideas of  
{\it basic products} (in free associative rings), Hall-Witt Theorems, 
identities, Baker-Campbell-Hausdorff formulae etc. 
 are presented beautifully
and in a very coherent manner in \cite{hall} and \cite{mks}. 

The
relationship with {\it Fox derivatives} is also clear from
\cite{mks}. Basics commutators are  used by
Gruenberg \cite{gruenberg} for his famous result that soluble  Engel
groups are locally nilpotent. 
Gruenberg shows that for a finitely generated and soluble group  
 the series of basic commutators eventually 
grow to  `Engel-like' and from this follows the nilpotency of the group. Whole
theories of basic commutators are presented in \cite{mward} which
haven't received the attention they should  and implicitly contain
further advancement and potential applications. A general `free system'
of free basic-commutator-like generators relating the 
free structure of terms of the lower
central series and related groups 
is given in \cite{hurward}; this follows from initial work of Ward
\cite{ward} and \cite{ward1}. They also appear quite frequently in
papers on Burnside's problem and many other areas.

 Is it time to `go back to basics'?

\subsection{Factor of a free abelian group by a free abelian group}

The basic commutators gives the structure of the free abelian group 
$\frac{\gamma_n(F)}{\gamma_{n+1}(F)} $  when $F$ is a free group and
are thus  a starting point for the study of lower central factors of
any group. 

From  

$\frac{\gamma_n(G)}{\gamma_{n+1}(G)} \cong
\frac{\frac{\gamma_n}{\gamma_{n+1}}}{\frac{R\cap \gamma_n}{R \cap
    \gamma_{n+1}}}$

it is seen {\it the structure of the lower central factors of $G$ is known
  once the structure of $\frac{R\cap \gamma_n}{R\cap
    \gamma_{n+1}}$, which is $ R\cap \ga_n \mod \ga_{n+1}$,
  is known.}

From now on write $\ga_nG$ for $\ga_n(G)$.
\subsection{Examples of use} 

\subsubsection{Free metabelian example}

Consider $R = \gamma_2\gamma_2F = F^{''}$, the second
derived group of $F$. We wish to know the structure of the lower
central factors of $F/F^{''}$.

It is relatively easy to show that if given $ a \in F^{''}$ then,
modulo  the $n^{th}$ term of the lower central series of $F$, $a$ is a
product of basic commutators of the form $[[...],[...]]$, that is, 
 each
basic commutator occurring in an element of $a\in F^{''}$ as the
product of basic commutators modulo $\ga_nF$  
has the `shape' of an element of $F{''}$.

Hence for this $R = F^{''}$,  $R\cap \gamma_nF$ modulo
$\gamma_{n+1}F$ is simply the free abelian group on the basic
commutators of weight $n$ of the form $[[...],[...]]$. 

Thus every element in the lower central factors of $F/F^{''}$ can be
written uniquely as a product of basic commutators of the form
$[.......]$ - no double brackets - and thus is free abelian on the
{\em simple basic commutators}, that is  basic commutators of the form
$[x_{i_1},x_{i_2},x_{i_3}, \ldots, x_{i_n}]$ with $ i_1 > i_2 \leq i_3
\ldots \leq i_n$. This is a result due to Magnus (unpublished) 
-- see Hanna Neumann's
book \cite{hneumann} pages 107-109  which has a different 
  and certainly longer proof.

\subsubsection{Other examples}
 
This method may also be used on other free groups in a variety. Take $R =
\gamma_2\gamma_3F$. (The group $F/R$ is the free group in the variety
abelian-by-(nilpotent of class $\leq 2$).)

Every element in this $R$ will be a product of basic commutators of
the form $[[\geq 3], [\geq 3]]$, that is those that have `shape' in
$\gamma_2\gamma_3F$; the others then will be a basis for the lower
central factors of $\frac{F}{\gamma_2\gamma_3F}$. So for example
$[[3],[2],\ldots,[2]]$ will be in this basis as well as  the
simple basic commutators.

In fact this trick works for any
$R = \gamma_{n_1}\gamma_{n_2}\gamma_{n_3}\ldots\gamma_{n_m}F$  in a 
{\em polynilpotent series} of $F$ for a
sequence $n_1, n_2, n_3 \ldots, n_m$. Included here would, for
example, 
be the derived series in which each $n_i=2$. The group $F/R$ for this $R$ is
known as the {\em free polynilpotent group} (relative to the sequence
$n_1,n_2,n_3,\ldots,n_m$).

Every element in $R$ is congruent modulo $\gamma_{n+1}F$ to a product
of basic commutators whose `shape' is in $R$ and the lower central
factors of $F/R$ is the free abelian group on the rest, that is  those
basic commutators whose
`shape' is less than (or not in) $R$. See Martin Ward \cite{mward} for
much more detail on this and many forms and shapes of basic
commutators and basic
sequences. 

\subsection{Further areas of application} 
The `shape' of a basic commutator has also been exploited in
identifications of various intersections in free groups and for 
groups determined by modules within the free  group ring; see for
example \cite{hur1},\cite{hur}, \cite{hur3} for solutions to the
{\it Fox-type and Lie Dimension 
problems} and other related identities in free groups, where basic-type
commutators come into play.

The method  also gives a way of identifying

$\gamma_{n_1}\gamma_{n_2}\gamma_{n_3}\ldots\gamma_{n_m}F \cap
\gamma_wF $ for any integer $w$. For example

$$ F^{''} \cap \gamma_wF = \prod_{i + j = w; i,j\geq 2}[\gamma_iF,
  \gamma_jF]$$

$$\gamma_2\gamma_3F \cap \gamma_wF = \prod_{i + j = w; i,j\geq
  3}[\gamma_iF, \gamma_jF]$$ $$ \gamma_3\gamma_2F \cap \gamma_wF =
\prod_{i + j + k = w; i,j,k \geq 2}[\gamma_iF, \gamma_jF, \gamma_kF]$$
 
\medskip 

or more generally:

\medskip

\begin{eqnarray*}
\lefteqn{ \gamma_{n_1}\gamma_{n_2}\gamma_{n_3}\ldots\gamma_{n_m}F \cap
  \gamma_wF = } \\ & & \\ & & \prod [(\gamma_{n_2}\gamma_{n_3} \ldots
  \gamma_{n_m}F \cap \gamma_{i_1}F), (\gamma_{n_2}\gamma_{n_3} \ldots
  \gamma_{n_m}F \cap \gamma_{i_2}F), \\ & & \ldots,
  (\gamma_{n_2}\gamma_{n_3} \ldots \gamma_{n_m}F \cap
  \gamma_{i_{n_1}}F)]
\end{eqnarray*}

where the product is over all integers $i_1, i_2, \ldots , i_{n_1}$
with $i_1 + i_2 + \ldots + i_{n_1} = w$

and each of $\gamma_{n_2}\gamma_{n_3}\ldots\gamma_{n_m}F \cap
\gamma_{i_j}F$ is determined by induction.
\medskip

\medskip
 
\section{Relative basic commutators}
 Suppose $F$ is finitely
generated and that $R$ is finitely generated as a normal subgroup. An
algorithm may be given for the determination of the structure of
$\frac{R \cap \gamma_nF}{R\cap \gamma_{n+1}F}$ in terms of free
generators of $\frac{\gamma_nF}{\gamma_{n+1}F}$; this gives an
algorithm for the determination of the the lower central factors of
$G$.

The way to do this is to construct {\em basic commutators relative to
  $R$ } in a process now to be defined.

Suppose $F$ is freely generated by $X = \{x_1, x_2, \ldots, x_n\}$ and
that $R$ is generated as a normal subgroup by $ A = \{r_1, r_2,
\ldots, r_m\}$. We assume no element of $A$ occurs in $X$ and order $A
\cup X$ by saying the elements of $A$ come after those of $X$.

Nielsen transformations were originally constructed to show that every
subgroup of a free group is free but has since found applications in
many areas. 

A {\em Nielsen Transformation} on a set $(a_i)_{i \in I}$ is one of
the following:
\begin{enumerate}
\item Exchange two of the $a_j$.
\item Replace an $a_j$ by $a_j^{-1}$.
\item Replace and $a_j$ by $a_ja_k, j \neq k$.
\item Carry out substitutions of types 1,2,3, repeatedly a finite
  number of times.

\end{enumerate}

See for example \cite{mks}, \cite{lynshupp} for further details on
Nielsen transformations.

A Nielsen transformation transforms a set of free generators onto a
set of free generators in such a way that the free group generated by
each set is the same.

Let $b$ be a basic commutator of weight $\geq$ 2 with $b = [a,c], a >
c$ for basic commutators $a,c$ and if $a = [d,e]$ then $e \leq c$. We
call $c$ the {\em second component } of $b$. If Nielsen
transformations are applied to a basic commutator we wish to define
the second component of such an expression.  Suppose then the second
components of $b$ and $b^{'}$ have been defined. In a transformation
of type 1 the second components are the same as the original second
components. We define the second component of $b^{-1}$ to be that of
$b$.


First let us consider $\frac{R}{R \cap \gamma_2F} \cong
\frac{R\gamma_2F}{\gamma_2F}$.  Let $H_1$ denote the subgroup of $F$
generated by $A$. Then we may construct a set of free generators
$x_{1_1}, x_{1_2}, \ldots , x_{1_{m_1}}$ for $F$ and a set of free
generators $y_{1_1}, y_{1_2}, \ldots , y_{1_{s_1}}$ for $H$ such that

$$y_{1_i} \equiv x_{1_i}^{d_{1_i}} \quad modulo \quad \gamma_2F, \quad
1\leq i \leq s_{1_1} \hspace{.2in} * $$
$$y_{1_i} \equiv 1 \quad modulo \quad \gamma_2F, \quad s_{1_1} < i
\leq s_{1_2} \hspace{.2in} **$$

where $0 < d_{1_i}$ divides $d_{1_{i+1}}$, and $s_{1_1} \leq m_1$.

This is done in the usual way by performing Nielsen Transformations
and reducing an $n\times m$ matrix of integers to diagonal from. Of course
this immediately gives the structure of $\frac{R}{R \cap \gamma_2F}
\cong \frac{R\gamma_2F}{\gamma_2F}$ as the direct product of infinite
cyclics and cyclics of orders $d_{1_i}$.

Note that $R = H_1^F$, the normal closure of $H_1$ in $F$.

Can we extend this to $\frac{R \cap \gamma_2F}{R \cap \gamma_3F}$ and
perhaps higher?

Consider first of all $\frac{R \cap \gamma_2F}{R \cap \gamma_3F}$.

What we do is construct {\em basic commutators relative to R} of
degree two in this case.

We define an $\overline{R}_2$ - basic to be an element of $\{[y_{1_i},
  x_{1_j}]; i > j\}$ $\cup$ $\{[x_{1_j}, y_{1_i}]; j > i \}$ for $ 1
\leq i \leq s_{1_i}$, and for all $ j, 1 \leq j \leq m_1$.

An $\overline{R}_2$ - basic is an element in $R \cap \gamma_2F$ and is
easy to show that the set of $\overline{R}_2$ - basics is linearly
independent modulo$ R \cap \gamma_3F$.

We haven't got a full generating set for $\frac{R \cap \gamma_2F}{R
  \cap \gamma_3F}$ as we have to include some more elements.

Let $H_2$ be the group defined by $\{\overline{R}_2$ - basics $\} $
$\cup \{y_{1_i}, \quad s_{1_1} < i \leq s_{1_2} \}$ $\cup \{[y_{1_i},
  x_{1_i}], \quad 1 \leq i \leq s_{1_1}\}$.

Then $H_2$ is a finitely generated subgroup of $\gamma_2F \cap R$. In
fact it can be shown that $H_2$ will generate $\frac{R \cap
  \gamma_2F}{R \cap \gamma_3F}$. We also know a free generating set
for $\frac{\gamma_2F}{\gamma_3F} $ as the free abelian group on the
basic commutators of weight 2.

Apply the diagonalisation process again to $H_2$ and
$\frac{\gamma_2F}{\gamma_3F} $ will give the structure of $\frac{R
  \cap \gamma_2F}{R \cap \gamma_3F}$.

Then there exists a set of free generators $x_{21}, x_{22}, \ldots
,x_{2m_2}$ for $\frac{\gamma_2F}{\gamma_3F}$ and a set of generators
$y_{21}, y_{22}, \ldots , y_{2s_{22}}$ for $H_2$ such that

$$y_{2i} \equiv x_{2i}^{d_{2i}} \quad modulo \quad \gamma_3F \quad for
\quad 1 \leq i \leq s_{21}$$

$$y_{2i} \equiv 1 \quad modulo \quad \gamma_3F \quad for \quad s_{21}
< i \leq s_{22}$$

Denote by $R_2$ the set of all $\{ y_{2i} / 1 \leq i \leq s_{21}
\}$. An element of $R_2$ will be called an $R-$basic of weight
2. Provided $\frac{R \cap \gamma_2F}{R \cap \gamma_3F}$ is generated
by $R_2$ this will give the structure of
$\frac{\frac{\gamma_2F}{\gamma_3F}}{\frac{R\cap \gamma_2F}{R\cap
    \gamma_3F}}$ and hence the structure of
$\frac{\gamma_2G}{\gamma_3G}$.

Of course this process can be continued and we can define a set of
$R-$basic of weight $n$ which will be a basis for $\frac{R \cap
  \gamma_nF}{R \cap \gamma_{n+1}F}$.

The process is really in a sense replacing a basic commutator which
corresponds non-trivially to a free generator modulo $\gamma_nF$ by
this free generator and consequently by any basic commutator which
contains this basic commutator as a constituent. The
following {\it basis theorem} follows from these constructions.

\begin{theorem} Every element $w$ in $R$ can be written uniquely
  in the form

$$ w \equiv r_1^{\alpha_1} r_2^{\alpha_2} \ldots r_t^{\alpha_t} \quad
  modulo \quad R \cap \gamma_{n+1}F$$

where the $r_1, r_2, \ldots ,r_t$ are the $R-$basic commutators of
weights $ \leq n$ and $r_1 < r_2 < \ldots < r_t$ and the $\alpha_i$
are integers.
\end{theorem}

\medskip

\begin{corollary} $R_i$ generates $\frac{R \cap \gamma_iF}{R \cap
    \gamma_{i+1}F}$ freely.
\end{corollary} 

This can be seen by noting that $R$ is generated by $\{r \cup [r,x]\}$
for any $r \in A$ and any $x \in F$.

\end{document}